\documentclass[fleqn,twoside]{article}%
\usepackage{amsfonts}
\usepackage{espcrc1}
\usepackage{amsmath}
\usepackage{amssymb}
\usepackage{graphicx}%
\newtheorem{theorem}{Theorem}

\newtheorem{corollary}[theorem]{Corollary}

\newtheorem{definition}[theorem]{Definition}

\newtheorem{lemma}[theorem]{Lemma}

\newenvironment{proof}[1][Proof]{\noindent\textbf{#1.} }{\ \rule{0.5em}{0.5em}}
\begin{document}
%
\title{Delayed Langevin type equations with two fractional derivatives}%
%

\author{N. I . Mahmudov\\
Department of Mathematics,
Eastern Mediterranean University\\
Famagusta, T.R. North Cyprus, via Mersin 10, Turkey}
\date{}%
%

\maketitle
%

\begin{abstract}
In this paper, we introduce a delayed Mittag-Leffler type function. With the
help of the delayed Mittag-Leffler type functions, we give an explicit formula
of solutions to linear nonhomogeneous fractional time-delay Langevin equations
involving two Riemann-Liouville fractional derivatives. The existence and
uniqueness of solutions are obtained by using an estimation of delayed
Mittag-Leffler type functions in terms of exponential functions and a weighted
norm via fixed point theorems. Further, we present Ulam--Hyers stability
results.%
\end{abstract}%


\section{Introduction}

One century ago, Langevin offered a detailed description of Brownian motion
due to collisions with much smaller fluid molecules. Many of the stochastic
problems in fluctuating media can be described by the Langevin equation
\cite{1}, \cite{2}. However, for some complex systems, the classical Langevin
equation cannot offer an adequate description of the problem. As a result,
various generalizations have been proposed that constitute the inadequacy of
the classical case and describe more physical phenomena in disordered regions
\cite{3}. One of them is so called fractional Langevin type equation, which is
obtained from the classical Langevin equation by replacing classical
derivative by fractional derivative. The nonlinear Langevin type equations
involving two fractional orders was introduced and investigated in
\cite{13}-\cite{18}.

Our goal is to study the following time-delay Langevin type equations with two
different fractional derivatives%
\begin{equation}
\left\{
\begin{tabular}
[c]{ll}%
$^{RL}D_{-h^{+}}^{\alpha}y\left(  t\right)  -\lambda\ ^{RL}D_{-h^{+}}^{\beta
}y\left(  t\right)  =\mu y\left(  t-h\right)  +f\left(  t,y\left(  t\right)
\right)  ,\ \ t\in\left(  0,T\right]  ,\ \ h>0,$ & \\
$y\left(  t\right)  =\varphi\left(  t\right)  ,\ \ \left(  I_{-h^{+}%
}^{2-\alpha}y\right)  \left(  -h^{+}\right)  =\varphi\left(  -h\right)
,\ \ -h\leq t\leq0,\ 0<\beta<1,\ 1<\alpha\leq2,\ \alpha-\beta>1,$ & \\
$^{RL}D_{-h^{+}}^{\alpha-1}y\left(  t\right)  =\ ^{RL}D_{-h^{+}}^{\alpha
-1}\varphi\left(  t\right)  ,$ $\ \ \ \ \left(  ^{RL}D_{-h+}^{\alpha
-1}y\right)  \left(  -h^{+}\right)  =\ \ ^{RL}D_{-h+}^{\alpha-1}\varphi\left(
-h\right)  ,$ &
\end{tabular}
\ \right.  \label{le2}%
\end{equation}
where $^{RL}D_{-h^{+}}^{\alpha}$ and $^{RL}D_{-h^{+}}^{\beta}$ are
Riemann-Liouville fractional derivatives of order $1<\alpha\leq2$ and
$0<\beta<1,$ $T=lh,$ $l\in\mathbb{N},$ $\varphi:\left[  -h,T\right]
\rightarrow\mathbb{R}$ is a two times continuously differentiable function.

The main contributions are as follows: (i) we introduce a delayed
Mittag-Leffler type function; (ii) we derive the representation of solutions
of nonhomogeneous equation (\ref{le2}); (iii) we give fundamental estimation
for delayed Mittag-Leffler type function in terms of exponential function;
(iv) we introduce a weighted norm (\ref{norm}) in $C(\left[  0,T\right]
,\mathbb{R}$) (the Banach space of all continuous functions from $\left[
0,T\right]  $ into $\mathbb{R}$ with the norm $\left\Vert x\right\Vert
_{\infty}=\sup\left\{  \left\vert x\left(  t\right)  \right\vert :t\in\left[
0,T\right]  \right\}  $) and establish sufficient conditions to guarantee the
existence and uniqueness of global solution on $\left[  0,T\right]  $ for
fractional time-delay Langevin type equation (\ref{le2}) and (v) we study the
Ulam-Hyers stability of (\ref{le2}) in a weighted space.

\section{Delayed Mittag-Leffler type function}

\begin{definition}
\label{def:01}\cite{kilbas} Mittag-Leffler type function of two parameters
$e_{\alpha,\beta}\left(  \lambda,\cdot\right)  :\mathbb{R}\rightarrow
\mathbb{R}$ is defined by%
\[
e_{\alpha,\beta}\left(  \lambda;t\right)  :=t^{\beta-1}E_{\alpha,\beta}\left(
\lambda t^{\alpha}\right)  :=t^{\beta-1}%
{\displaystyle\sum\limits_{k=0}^{\infty}}
\frac{\lambda^{k}t^{\alpha k}}{\Gamma\left(  k\alpha+\beta\right)
},\ \ \ \alpha,\beta>0,\ t\in\mathbb{R}.
\]

\end{definition}

\begin{definition}
\label{def:02}\cite{kilbas} Let $\lambda_{l},b_{j}\in C,$ $\alpha_{l}%
,\beta_{j}\in\mathbb{R},\ l=1,...,p;$ $j=1,...,q$. Generalized Wright function
$_{p}\Psi_{q}\left(  \cdot\right)  :C\rightarrow C$ is defined by%
\[
_{p}\Psi_{q}\left(  z\right)  =\ _{p}\Psi_{q}\left[  \left.
\begin{array}
[c]{c}%
\left(  \lambda_{l},\alpha_{l}\right)  _{1,p}\\
\left(  b_{l},\beta_{l}\right)  _{1,q}%
\end{array}
\right\vert z\right]  =%
{\displaystyle\sum_{k=0}^{\infty}}
\frac{%
{\displaystyle\prod_{l=1}^{p}}
\Gamma\left(  \lambda_{l}+\alpha_{l}k\right)  }{%
{\displaystyle\prod_{j=1}^{q}}
\Gamma\left(  b_{j}+\beta_{j}k\right)  }\frac{z^{k}}{k!}%
\]

\end{definition}

It is known that, see \cite{kilbas} Theorem 1.5, if $%
{\displaystyle\sum_{j=1}^{q}}
\beta_{j}-%
{\displaystyle\sum_{l=1}^{p}}
\alpha_{l}>-1$, then the series $_{p}\Psi_{q}\left(  z\right)  $ is absolutely
convergent for all $z\in C$.

The following function is defined in \cite{kilbas}
\[
G_{\alpha,\beta}\left(  \lambda,\mu;t\right)  =%
{\displaystyle\sum_{n=0}^{\infty}}
\frac{\lambda^{n}}{n!}t^{\alpha n}\ _{1}\Psi_{1}\left[  \left.
\begin{array}
[c]{c}%
\left(  n+1,1\right) \\
\left(  \alpha n+\alpha,\alpha-\beta\right)
\end{array}
\right\vert \mu t^{\alpha-\beta}\right]  =%
{\displaystyle\sum_{n=0}^{\infty}}
{\displaystyle\sum_{k=0}^{\infty}}
\left(
\begin{array}
[c]{c}%
n+k\\
k
\end{array}
\right)  \lambda^{n}\mu^{k}\frac{t^{\alpha n+k\left(  \alpha-\beta\right)  }%
}{\Gamma\left(  \alpha n+\left(  \alpha-\beta\right)  k+\alpha\right)  }.
\]

Motivated by this we introduce so called delayed Mittag-Leffler type function
generated by $\lambda,\mu$ as follows:

\begin{definition}
\label{def:03}Let $\alpha,\beta,\gamma>0.$ Delayed Mittag-Leffler type
function generated by $\lambda,\mu$ of two parameters $\mathfrak{E}%
_{\alpha,\beta}^{h,\gamma}\left(  \lambda,\mu;\cdot\right)  :\mathbb{R}%
\rightarrow\mathbb{R}$ is defined by%
\[
\mathfrak{E}_{\alpha,\beta}^{h,\gamma}\left(  \lambda,\mu;t\right)  =%
{\displaystyle\sum_{n=0}^{\infty}}
{\displaystyle\sum_{k=0}^{\infty}}
\left(
\begin{array}
[c]{c}%
n+k\\
k
\end{array}
\right)  \lambda^{n}\mu^{k}\dfrac{\left(  t-kh\right)  ^{k\gamma+n\alpha
+\beta-1}}{\Gamma\left(  k\gamma+n\alpha+\beta\right)  }H\left(  t-kh\right)
,
\]
where $H$ is the Heaviside function
\[
H\left(  t\right)  =\left\{
\begin{array}
[c]{c}%
1,\ \ t\geq0,\\
0,\ \ \ t<0.
\end{array}
\right.
\]

\end{definition}

\begin{definition}
\label{def:11}\cite{li1} Delayed Mittag-Leffler type function of two
parameters $E^{h,\alpha,\beta}\left(  \mu;\cdot\right)  :\mathbb{R}%
\rightarrow\mathbb{R}$ is defined by%
\begin{equation}
E^{h,\alpha,\beta}\left(  \mu;t\right)  :=\left\{
\begin{tabular}
[c]{ll}%
$\Theta,$ & $-\infty<t\leq-h,$\\
$I\frac{\left(  h+t\right)  ^{\beta-1}}{\Gamma\left(  \beta\right)  },$ &
$-h<t\leq0,$\\
$I\frac{\left(  h+t\right)  ^{\beta-1}}{\Gamma\left(  \beta\right)  }+\mu
\frac{t^{\alpha+\beta-1}}{\Gamma\left(  \alpha+\beta\right)  }+\mu^{2}%
\frac{\left(  t-h\right)  ^{2\alpha+\beta-1}}{\Gamma\left(  2\alpha
+\beta\right)  }+...+\mu^{k}\frac{\left(  t-\left(  k-1\right)  h\right)
^{k\alpha+\beta-1}}{\Gamma\left(  k\alpha+\beta\right)  },$ & $\left(
k-1\right)  h<t\leq kh.$%
\end{tabular}
\ \ \ \ \ \ \ \ \ \ \ \ \ \right.  \label{ml2}%
\end{equation}

\end{definition}

\begin{lemma}
We have

\begin{description}
\item[(i)] if $\mu=0$, then $\mathfrak{E}_{\alpha-\beta,\beta}^{h,\gamma
}\left(  \lambda,\mu;t\right)  =%
{\displaystyle\sum_{n=0}^{\infty}}
\lambda^{n}\dfrac{t^{n\left(  \alpha-\beta\right)  +\beta-1}}{\Gamma\left(
n\left(  \alpha-\beta\right)  +\beta\right)  }H\left(  t\right)  =t^{\beta
-1}E_{\alpha-\beta,\beta}\left(  \lambda;t\right)  =e_{\alpha-\beta,\beta
}\left(  \lambda;t\right)  .$

\item[(ii)] if $\lambda=0$, then $\mathfrak{E}_{\alpha,\beta}^{h,\gamma
}\left(  \lambda,\mu;t\right)  =%
{\displaystyle\sum_{k=0}^{\infty}}
\mu^{k}\dfrac{\left(  t-kh\right)  ^{k\gamma+\beta-1}}{\Gamma\left(
k\gamma+\beta\right)  }H\left(  t-kh\right)  =E^{h,\gamma,\beta}\left(
\mu;t-h\right)  .$
\end{description}
\end{lemma}

\begin{lemma}
\label{Lem:2}$\mathfrak{E}_{\alpha-\beta,\alpha}^{h,\alpha-j}\left(
\lambda,\mu;t\right)  $, $j=0,1,$ satisfies the following equation%
\[
^{RL}D_{-h^{+}}^{\alpha}\mathfrak{E}_{\alpha-\beta,\alpha}^{h,\alpha-j}\left(
\lambda,\mu;t\right)  -\lambda\ ^{RL}D_{-h^{+}}^{\beta}\mathfrak{E}%
_{\alpha-\beta,\alpha}^{h,\alpha-j}\left(  \lambda,\mu;t\right)
=\mu\mathfrak{E}_{\alpha-\beta,\alpha}^{h,\alpha-j}\left(  \lambda
,\mu;t-h\right)  .
\]

\end{lemma}

\begin{proof}
We prove lemma just for $j=0$. Direct calculations show that%
\begin{align*}
\mathfrak{E}_{\alpha-\beta,\alpha}^{h,\alpha}\left(  \lambda,\mu;t\right)   &
=%
{\displaystyle\sum_{n=0}^{\infty}}
{\displaystyle\sum_{k=0}^{\infty}}
\left(
\begin{array}
[c]{c}%
n+k\\
k
\end{array}
\right)  \lambda^{n}\mu^{k}\dfrac{\left(  t-kh\right)  ^{k\alpha+n\left(
\alpha-\beta\right)  +\alpha-1}}{\Gamma\left(  k\alpha+n\left(  \alpha
-\beta\right)  +\alpha\right)  }H\left(  t-kh\right)  ,\\
^{RL}D_{-h^{+}}^{\beta}\mathfrak{E}_{\alpha-\beta,\alpha}^{h,\alpha}\left(
\lambda,\mu;t\right)   &  =%
{\displaystyle\sum_{n=0}^{\infty}}
{\displaystyle\sum_{k=0}^{\infty}}
\left(
\begin{array}
[c]{c}%
n+k\\
k
\end{array}
\right)  \lambda^{n}\mu^{k}\dfrac{\left(  t-kh\right)  ^{k\alpha+n\left(
\alpha-\beta\right)  +\alpha-\beta-1}}{\Gamma\left(  k\alpha+n\left(
\alpha-\beta\right)  +\alpha-\beta\right)  }H\left(  t-kh\right)  ,\\
^{RL}D_{-h^{+}}^{\alpha}\mathfrak{E}_{\alpha-\beta,\alpha}^{h,\alpha}\left(
\lambda,\mu;t\right)   &  =%
{\displaystyle\sum_{n=1}^{\infty}}
\lambda^{n}\dfrac{t^{n\left(  \alpha-\beta\right)  -1}}{\Gamma\left(  n\left(
\alpha-\beta\right)  \right)  }+%
{\displaystyle\sum_{n=0}^{\infty}}
{\displaystyle\sum_{k=1}^{\infty}}
\left(
\begin{array}
[c]{c}%
n+k\\
k
\end{array}
\right)  \lambda^{n}\mu^{k}\dfrac{\left(  t-kh\right)  ^{k\alpha+n\left(
\alpha-\beta\right)  -1}}{\Gamma\left(  k\alpha+n\left(  \alpha-\beta\right)
\right)  }H\left(  t-kh\right)  .
\end{align*}
Using the binomial identity $\left(
\begin{array}
[c]{c}%
n+k\\
k
\end{array}
\right)  =\left(
\begin{array}
[c]{c}%
n+k-1\\
k
\end{array}
\right)  +\left(
\begin{array}
[c]{c}%
n+k-1\\
k-1
\end{array}
\right)  $ we get%
\begin{align*}
&  ^{RL}D_{-h^{+}}^{\alpha}\mathfrak{E}_{\alpha-\beta,\alpha}^{h,\alpha
}\left(  \lambda,\mu;t\right) \\
&  =%
{\displaystyle\sum_{n=1}^{\infty}}
\lambda^{n}\dfrac{t^{n\left(  \alpha-\beta\right)  -1}}{\Gamma\left(  n\left(
\alpha-\beta\right)  \right)  }+%
{\displaystyle\sum_{n=0}^{\infty}}
{\displaystyle\sum_{k=1}^{\infty}}
\left(
\begin{array}
[c]{c}%
n+k-1\\
k
\end{array}
\right)  \lambda^{n}\mu^{k}\dfrac{\left(  t-kh\right)  ^{k\alpha+n\left(
\alpha-\beta\right)  -1}}{\Gamma\left(  k\alpha+n\left(  \alpha-\beta\right)
\right)  }H\left(  t-kh\right) \\
&  +%
{\displaystyle\sum_{n=0}^{\infty}}
{\displaystyle\sum_{k=1}^{\infty}}
\left(
\begin{array}
[c]{c}%
n+k-1\\
k-1
\end{array}
\right)  \lambda^{n}\mu^{k}\dfrac{\left(  t-kh\right)  ^{k\alpha+n\left(
\alpha-\beta\right)  -1}}{\Gamma\left(  k\alpha+n\left(  \alpha-\beta\right)
\right)  }H\left(  t-kh\right) \\
&  =\lambda%
{\displaystyle\sum_{n=0}^{\infty}}
\lambda^{n}\dfrac{t^{n\left(  \alpha-\beta\right)  +\alpha-\beta-1}}%
{\Gamma\left(  n\left(  \alpha-\beta\right)  +\alpha-\beta\right)  }+%
{\displaystyle\sum_{n=0}^{\infty}}
{\displaystyle\sum_{k=1}^{\infty}}
\left(
\begin{array}
[c]{c}%
n+k\\
k
\end{array}
\right)  \lambda^{n+1}\mu^{k}\dfrac{\left(  t-kh\right)  ^{k\alpha+n\left(
\alpha-\beta\right)  +\alpha-\beta-1}}{\Gamma\left(  k\alpha+n\left(
\alpha-\beta\right)  +\alpha-\beta\right)  }H\left(  t-kh\right)  +\\
&  +%
{\displaystyle\sum_{n=0}^{\infty}}
{\displaystyle\sum_{k=0}^{\infty}}
\left(
\begin{array}
[c]{c}%
n+k\\
k
\end{array}
\right)  \lambda^{n}\mu^{k+1}\dfrac{\left(  t-h-kh\right)  ^{k\alpha+n\left(
\alpha-\beta\right)  +\alpha-1}}{\Gamma\left(  k\alpha+n\left(  \alpha
-\beta\right)  +\alpha\right)  }H\left(  t-h-kh\right) \\
&  =\lambda\ ^{RL}D_{-h^{+}}^{\beta}\mathfrak{E}_{\alpha-\beta,\alpha
}^{h,\alpha}\left(  \lambda,\mu;t\right)  +\mu\mathfrak{E}_{\alpha
-\beta,\alpha}^{h,\alpha}\left(  \lambda,\mu;t-h\right)  .
\end{align*}
Lemma is proved.
\end{proof}

\begin{theorem}
\label{thm:2} A solution $y\in C\left(  \left(  \left(  p-1\right)
h,ph\right]  ,\mathbb{R}\right)  ,$ $0\leq p\leq l,$ of (\ref{le2}) with $f=0$
has a form%
\begin{align*}
y\left(  t\right)   &  =\mathfrak{E}_{\alpha-\beta,\alpha}^{h,\alpha}\left(
\lambda,\mu;t+h\right)  \ ^{RL}D_{-h^{+}}^{\alpha-1}\varphi\left(  -h\right)
+\mathfrak{E}_{\alpha-\beta,\alpha}^{h,\alpha-1}\left(  \lambda,\mu
;t+h\right)  \left(  I_{-h^{+}}^{2-\alpha}\varphi\right)  \left(
-h^{+}\right) \\
&  +\int_{-h}^{t}\mathfrak{E}_{\alpha-\beta,\alpha}^{h,\alpha}\left(
\lambda,\mu;t-s\right)  \left(  \ ^{RL}D_{-h^{+}}^{\alpha}\varphi\left(
s\right)  -\lambda\ ^{RL}D_{-h^{+}}^{\beta}\varphi\left(  s\right)  \right)
ds.
\end{align*}

\end{theorem}

\begin{proof}
We are looking for a solution of the form%
\begin{align}
y\left(  t\right)   &  =\mathfrak{E}_{\alpha-\beta,\alpha}^{h,\alpha}\left(
\lambda,\mu;t+h\right)  c_{1}+\mathfrak{E}_{\alpha-\beta,\alpha}^{h,\alpha
-1}\left(  \lambda,\mu;t+h\right)  c_{2}+\int_{-h}^{t}\mathfrak{E}%
_{\alpha-\beta,\alpha}^{h,\alpha}\left(  \lambda,\mu;t-s\right)  g\left(
s\right)  ds,\nonumber\\
y\left(  t\right)   &  =\varphi\left(  t\right)  ,\ \ \left(  I_{-h^{+}%
}^{2-\alpha}y\right)  \left(  -h^{+}\right)  =\varphi\left(  -h\right)
,\ \ -h\leq t\leq0,\ \label{rep1}\\
^{RL}D_{-h^{+}}^{\alpha-1}y\left(  t\right)   &  =\ ^{RL}D_{-h^{+}}^{\alpha
-1}\varphi\left(  t\right)  ,\ \ \ \ \left(  ^{RL}D_{-h+}^{\alpha-1}y\right)
\left(  -h^{+}\right)  =\ \ ^{RL}D_{-h+}^{\alpha-1}\varphi\left(  -h\right)
,\nonumber
\end{align}
where $c_{1}$ and $c_{2}$ are unknown constants, $g(t)$ is an unknown
continuously differentiable function. Direct calculations and Lemma
\ref{Lem:2} show that%
\begin{gather*}
^{RL}D_{-h^{+}}^{\alpha-j}\left(  \int_{-h}^{\cdot}\mathfrak{E}_{\alpha
-\beta,\alpha}^{h,\alpha}\left(  \lambda,\mu;\cdot-s\right)  g\left(
s\right)  ds\right)  \left(  -h^{+}\right)  =0,\ \ j=0,1,\\
^{RL}D_{-h^{+}}^{\alpha}\left(  \int_{-h}^{\cdot}\mathfrak{E}_{\alpha
-\beta,\alpha}^{h,\alpha}\left(  \lambda,\mu;\cdot-s\right)  g\left(
s\right)  ds\right)  \left(  t\right)  =\lambda\int_{-h}^{t}\mathfrak{E}%
_{\alpha-\beta,\alpha}^{h,\alpha}\left(  \lambda,\mu;t-s\right)  g\left(
s\right)  ds+g\left(  t\right)  ,\\
^{RL}D_{-h^{+}}^{\alpha}\mathfrak{E}_{\alpha-\beta,\alpha}^{h,\alpha-j}\left(
\lambda,\mu;t+h\right)  =\lambda\ ^{RL}D_{-h^{+}}^{\beta}\mathfrak{E}%
_{\alpha-\beta,\alpha}^{h,\alpha-j}\left(  \lambda,\mu;t+h\right)  ,\ \ j=0,1.
\end{gather*}
Taking Riemann-Liouville derivative $^{RL}D_{-h^{+}}^{\alpha-2}y\left(
-h^{+}\right)  $ and $^{RL}D_{-h^{+}}^{\alpha-1}y\left(  -h^{+}\right)  $ of
$y\left(  t\right)  $ at $-h$ we get
\begin{align*}
^{RL}D_{-h^{+}}^{\alpha-2}y\left(  -h^{+}\right)   &  =c_{2}=\left(
I_{-h^{+}}^{2-\alpha}\varphi\right)  \left(  -h^{+}\right)  ,\\
^{RL}D_{-h^{+}}^{\alpha-1}y\left(  -h^{+}\right)   &  =c_{1}=\ ^{RL}D_{-h^{+}%
}^{\alpha-1}\varphi\left(  -h^{+}\right)
\end{align*}
On the other hand differentiating (\ref{rep1}), we obtain%
\begin{align*}
\ \theta^{RL}D_{-h^{+}}^{\alpha}\varphi\left(  t\right)   &  =\ ^{RL}%
D_{-h^{+}}^{\alpha}y\left(  t\right)  =\lambda\ ^{RL}D_{-h^{+}}^{\beta
}\mathfrak{E}_{\alpha-\beta,\alpha}^{h,\alpha}\left(  \lambda,\mu;t+h\right)
c_{1}\\
&  +\lambda\ ^{RL}D_{-h^{+}}^{\beta}\mathfrak{E}_{\alpha-\beta,\alpha
}^{h,\alpha-1}\left(  \lambda,\mu;t+h\right)  c_{2}\\
&  +\lambda\int_{-h}^{t}\ ^{RL}D_{-h^{+}}^{\beta}\mathfrak{E}_{\alpha
-\beta,\alpha}^{h,\alpha}\left(  \lambda,\mu;t-s\right)  g\left(  s\right)
ds+g\left(  t\right)  ,\ \ \ -h\leq t\leq0.
\end{align*}
Therefore, $g\left(  t\right)  =\ ^{RL}D_{-h^{+}}^{\alpha}\varphi\left(
t\right)  -\lambda\ ^{RL}D_{-h^{+}}^{\beta}\varphi\left(  t\right)  $ and the
desired result holds.
\end{proof}

\begin{theorem}
\label{thm:11}The solution $y(t)$ of (\ref{le2}) satisfying zero initial
condition $y\left(  t\right)  =0,\ -h\leq t\leq0,$ has a form%
\[
y\left(  t\right)  =\int_{0}^{t}\mathfrak{E}_{\alpha-\beta,\alpha}^{h,\alpha
}\left(  \lambda,\mu;t-s\right)  f\left(  s\right)  ds,\ \ t\geq0.
\]

\end{theorem}

\begin{proof}
Taking the fractional derivative $\ ^{RL}D_{-h^{+}}^{\alpha}$ of $\int_{0}%
^{t}\mathfrak{E}_{\alpha-\beta,\alpha}^{h,\alpha}\left(  \lambda
,\mu;t-s\right)  f\left(  s\right)  ds$ we can easily get the result.
\end{proof}

Combining Theorems \ref{thm:2} and \ref{thm:11}, we have the following result.

\begin{corollary}
A solution $y\in C\left(  \left[  -h,T\right]  \cap\left(  \left(  p-1\right)
h,ph\right]  ,\mathbb{R}\right)  ,$ $p\in\mathbb{N},$ of (\ref{le2}) has a
form%
\begin{align*}
y\left(  t\right)   &  =\mathfrak{E}_{\alpha-\beta,\alpha}^{h,\alpha}\left(
\lambda,\mu;t+h\right)  \ ^{RL}D_{-h^{+}}^{\alpha-1}\varphi\left(  -h\right)
+\mathfrak{E}_{\alpha-\beta,\alpha}^{h,\alpha-1}\left(  \lambda,\mu
;t+h\right)  \left(  I_{-h^{+}}^{2-\alpha}\varphi\right)  \left(
-h^{+}\right) \\
&  +\int_{-h}^{t}\mathfrak{E}_{\alpha-\beta,\alpha}^{h,\alpha}\left(
\lambda,\mu;t-s\right)  \left(  \ ^{RL}D_{-h^{+}}^{\alpha}\varphi\left(
s\right)  -\lambda\ ^{RL}D_{-h^{+}}^{\beta}\varphi\left(  s\right)  \right)
ds+\int_{0}^{t}\mathfrak{E}_{\alpha-\beta,\alpha}^{h,\alpha}\left(
\lambda,\mu;t-s\right)  f\left(  s\right)  ds.
\end{align*}

\end{corollary}

\section{Existence, uniqueness and Ulam-Hyers stability}

To introduce a fixed point problem associated with (\ref{le2}) we define an
integral operator $\mathfrak{F:}C\left(  \left[  0,T\right]  ,\mathbb{R}%
\right)  \rightarrow C\left(  \left[  0,T\right]  ,\mathbb{R}\right)  $ by%
\begin{align*}
\left(  \mathfrak{F}y\right)  \left(  t\right)   &  :=\mathfrak{E}%
_{\alpha-\beta,\alpha}^{h,\alpha}\left(  \lambda,\mu;t+h\right)
\ ^{RL}D_{-h^{+}}^{\alpha-1}\varphi\left(  -h\right)  +\mathfrak{E}%
_{\alpha-\beta,\alpha}^{h,\alpha-1}\left(  \lambda,\mu;t+h\right)  \left(
I_{-h^{+}}^{2-\alpha}\varphi\right)  \left(  -h^{+}\right) \\
&  +\int_{-h}^{t}\mathfrak{E}_{\alpha-\beta,\alpha}^{h,\alpha}\left(
\lambda,\mu;t-s\right)  \left(  \ ^{RL}D_{-h^{+}}^{\alpha}\varphi\left(
s\right)  -A\varphi\left(  s\right)  \right)  ds+\int_{0}^{t}\mathfrak{E}%
_{\alpha-\beta,\alpha}^{h,\alpha}\left(  \lambda,\mu;t-s\right)  f\left(
s,y\left(  s\right)  \right)  ds
\end{align*}

\begin{lemma}
\label{Lem:1}We have%
\begin{equation}
\left\vert \mathfrak{E}_{\alpha-\beta,\alpha}^{h,\alpha}\left(  \lambda
,\mu;t\right)  \right\vert \leq t^{\alpha-1}\exp\left(  \left\vert
\lambda\right\vert t^{\alpha-\beta}+\left\vert \mu\right\vert t^{\alpha
}\right)  . \label{es}%
\end{equation}

\end{lemma}

\begin{proof}
Firstly, we estimate $\mathfrak{E}_{\alpha-\beta,\alpha}^{h,\alpha}\left(
\lambda,\mu;t\right)  $ as follows.
\begin{align}
\left\vert \mathfrak{E}_{\alpha-\beta,\alpha}^{h,\alpha}\left(  \lambda
,\mu;t\right)  \right\vert  &  \leq%
{\displaystyle\sum_{n=0}^{\infty}}
{\displaystyle\sum_{k=0}^{\infty}}
\left(
\begin{array}
[c]{c}%
n+k\\
k
\end{array}
\right)  \left\vert \lambda\right\vert ^{n}\left\vert \mu\right\vert
^{k}\dfrac{\left(  t-kh\right)  ^{k\alpha+n\left(  \alpha-\beta\right)
+\alpha-1}}{\Gamma\left(  k\alpha+n\left(  \alpha-\beta\right)  +\alpha
\right)  }H\left(  t-kh\right) \nonumber\\
&  \leq%
{\displaystyle\sum_{n=0}^{\infty}}
{\displaystyle\sum_{k=0}^{\infty}}
\frac{\left(  n+k\right)  !}{n!k!}\left\vert \lambda\right\vert ^{n}\left\vert
\mu\right\vert ^{k}\dfrac{t^{k\alpha+n\left(  \alpha-\beta\right)  +\alpha-1}%
}{\Gamma\left(  k\alpha+n\left(  \alpha-\beta\right)  +\alpha\right)  }.
\label{es1}%
\end{align}
Since $1<\alpha<2,$ $\alpha-\beta>1$ we have%
\begin{equation}
\Gamma\left(  k\alpha+n\left(  \alpha-\beta\right)  +\alpha\right)
>\Gamma\left(  k+n+1\right)  . \label{es2}%
\end{equation}
From estimations (\ref{es1}) and (\ref{es2}), it follows that
\end{proof}

\begin{align*}
\left\vert \mathfrak{E}_{\alpha-\beta,\alpha}^{h,\alpha}\left(  \lambda
,\mu;t\right)  \right\vert  &  \leq%
{\displaystyle\sum_{n=0}^{\infty}}
{\displaystyle\sum_{k=0}^{\infty}}
\frac{\left(  n+k\right)  !}{n!k!}\left\vert \lambda\right\vert ^{n}\left\vert
\mu\right\vert ^{k}\dfrac{t^{k\alpha+n\left(  \alpha-\beta\right)  +\alpha-1}%
}{\left(  n+k\right)  !}\\
&  =%
{\displaystyle\sum_{n=0}^{\infty}}
{\displaystyle\sum_{k=0}^{\infty}}
\left\vert \lambda\right\vert ^{n}\left\vert \mu\right\vert ^{k}%
\dfrac{t^{k\alpha+n\left(  \alpha-\beta\right)  +\alpha-1}}{n!k!}\\
&  =t^{\alpha-1}%
{\displaystyle\sum_{n=0}^{\infty}}
\left\vert \lambda\right\vert ^{n}\dfrac{t^{n\left(  \alpha-\beta\right)  }%
}{n!}%
{\displaystyle\sum_{k=0}^{\infty}}
\left\vert \mu\right\vert ^{k}\dfrac{t^{k\alpha}}{k!}=t^{\alpha-1}\exp\left(
\left\vert \lambda\right\vert t^{\alpha-\beta}+\left\vert \mu\right\vert
t^{\alpha}\right)  .
\end{align*}

\begin{theorem}
\label{Thm:1}Let $f:\left[  0,T\right]  \times\mathbb{R}\rightarrow\mathbb{R}$
be a continuous function such that the following conditions hold:

\begin{enumerate}
\item[(A$_{1}$)] there exists $L_{f}>0$ such that%
\[
\left\vert f\left(  t,y\right)  -f\left(  t,z\right)  \right\vert \leq
L_{f}\left\vert y-z\right\vert ,\ \ \forall\left(  t,y\right)  ,\left(
t,z\right)  \in\left[  0,T\right]  \times\mathbb{R}.
\]
\newline Then the problem (\ref{le2}) has a unique solution in $C\left(
\left[  0,T\right]  ,\mathbb{R}\right)  .$
\end{enumerate}
\end{theorem}

\begin{proof}
We will apply Contraction mapping principle to show that $\mathfrak{F}$ has a
unique fixed point. At first glance it seems natural to use the maximum norm
on $C\left(  \left[  0,T\right]  ,\mathbb{R}\right)  $, but this choice would
lead us only to a local solution defined on a subinterval of $\left[
0,T\right]  $. The trick is to use the weighted maximum norm%
\begin{equation}
\left\Vert y\right\Vert _{\omega}:=\max\left\{  \frac{\left\vert y\left(
t\right)  \right\vert }{E_{\alpha}\left(  \omega;t\right)  }:0\leq t\leq
T\right\}  \label{norm}%
\end{equation}
on $C\left(  \left[  0,T\right]  ,\mathbb{R}\right)  $. Observe that $C\left(
\left[  0,T\right]  ,\mathbb{R}\right)  $ is a Banach space with this norm
since it is equivalent to the maximum norm.

We now show that $\mathfrak{F}$ is a contraction on $\left(  C\left(  \left[
0,T\right]  ,\mathbb{R}\right)  ,\left\Vert \cdot\right\Vert _{\omega}\right)
$. To see this let $y,z\in C\left(  \left[  0,T\right]  ,\mathbb{R}\right)  $
and notice that%
\[
\left(  \mathfrak{F}y\right)  \left(  t\right)  -\left(  \mathfrak{F}z\right)
\left(  t\right)  =\int_{0}^{t}\mathfrak{E}_{\alpha-\beta,\alpha}^{h,\alpha
}\left(  \lambda,\mu;t-s\right)  \left(  f\left(  s,y\left(  s\right)
\right)  -f\left(  s,\left(  s\right)  \right)  \right)  ds.
\]
Thus for $t\in\left[  0,T\right]  ,$ from Lemma \ref{Lem:1} and Lipschitz
condition (A$_{1}$) it follows that
\begin{align*}
&  \frac{1}{E_{\alpha}\left(  \omega;t\right)  }\left\vert \int_{0}%
^{t}\mathfrak{E}_{\alpha-\beta,\alpha}^{h,\alpha}\left(  \lambda
,\mu;t-s\right)  \left(  f\left(  s,y\left(  s\right)  \right)  -f\left(
s,\left(  s\right)  \right)  \right)  ds\right\vert \\
&  \leq L_{f}\frac{1}{E_{\alpha}\left(  \omega;t\right)  }\int_{0}^{t}\left(
t-s\right)  ^{\alpha-1}\exp\left(  \left\vert \lambda\right\vert \left(
t-s\right)  ^{\alpha-\beta}\right)  \exp\left(  \left\vert \mu\right\vert
\left(  t-s\right)  ^{\alpha}\right)  \left\vert y\left(  s\right)  -z\left(
s\right)  \right\vert ds\\
&  \leq L_{f}\exp\left(  \left\vert \lambda\right\vert t^{\alpha-\beta
}+\left\vert \mu\right\vert t^{\alpha}\right)  \frac{1}{E_{\alpha}\left(
\omega;t\right)  }\int_{0}^{t}\left(  t-s\right)  ^{\alpha-1}\frac{E_{\alpha
}\left(  \omega;s\right)  }{E_{\alpha}\left(  \omega;s\right)  }\left\vert
y\left(  s\right)  -z\left(  s\right)  \right\vert ds\\
&  \leq L_{f}\exp\left(  \left\vert \lambda\right\vert t^{\alpha-\beta
}+\left\vert \mu\right\vert t^{\alpha}\right)  \frac{1}{E_{\alpha}\left(
\omega;t\right)  }\int_{0}^{t}\left(  t-s\right)  ^{\alpha-1}E_{\alpha}\left(
\omega;s\right)  ds\left\Vert y-z\right\Vert _{\omega}.
\end{align*}
On the other hand, it is known that%
\[
\frac{\omega}{\Gamma\left(  \alpha\right)  }\int_{0}^{t}\left(  t-s\right)
^{\alpha-1}E_{\alpha}\left(  \omega;s\right)  ds=E_{\alpha}\left(
\omega;t\right)  -1\leq E_{\alpha}\left(  \omega;t\right)  .
\]
Combining the last two inequalities we get
\begin{equation}
\frac{1}{E_{\alpha}\left(  \omega;t\right)  }\left\vert \int_{0}%
^{t}\mathfrak{E}_{\alpha-\beta,\alpha}^{h,\alpha}\left(  \lambda
,\mu;t-s\right)  \left(  f\left(  s,y\left(  s\right)  \right)  -f\left(
s,z\left(  s\right)  \right)  \right)  ds\right\vert \leq\frac{\Gamma\left(
\alpha\right)  }{\omega}L_{f}\exp\left(  \left\vert \lambda\right\vert
T^{\alpha-\beta}+\left\vert \mu\right\vert T^{\alpha}\right)  \left\Vert
y-z\right\Vert _{\omega}. \label{qw2}%
\end{equation}
Taking maximum over $\left[  0,T\right]  $ we get%
\[
\left\Vert \mathfrak{F}y-\mathfrak{F}z\right\Vert _{\omega}\leq\frac
{\Gamma\left(  \alpha\right)  }{\omega}L_{f}\exp\left(  \left\vert
\lambda\right\vert T^{\alpha-\beta}+\left\vert \mu\right\vert T^{\alpha
}\right)  \left\Vert y-z\right\Vert _{\omega}.
\]
Choose $\omega>0$ so that $\omega>L_{f}\Gamma\left(  \alpha\right)
\exp\left(  \left\vert \lambda\right\vert T^{\alpha-\beta}+\left\vert
\mu\right\vert T^{\alpha}\right)  .\ $So $\mathfrak{F}$ is a contraction. Thus
by the Banach fixed point theorem has $\mathfrak{F}$ a unique fixed point in
$C\left(  \left[  0,T\right]  ,\mathbb{R}\right)  $.
\end{proof}

Next, we are going to discuss Ulam--Hyers stability of the equation
(\ref{le2}) on the time interval $\left[  0,T\right]  $.

Let $\varepsilon>0$.. Consider the equation (\ref{le2}) and the inequality%
\begin{equation}
\left\vert ^{RL}D_{-h^{+}}^{\alpha}x\left(  t\right)  -\lambda\ ^{RL}%
D_{-h^{+}}^{\beta}x\left(  t\right)  -\mu x\left(  t-h\right)  -f\left(
t,x\left(  t\right)  \right)  \right\vert \leq\varepsilon,\ \ \ t\in\left[
0,T\right]  . \label{ul1}%
\end{equation}

\begin{definition}
We say that the equation (\ref{le2}) is Ulam--Hyers stable if there exists
$c>0$ such that for each $\varepsilon>0$ and for each solution $x\in C\left(
\left[  0,T\right]  ,\mathbb{R}\right)  $ of the inequality (\ref{ul1}) there
exists a solution $y\in C\left(  \left[  0,T\right]  ,\mathbb{R}\right)  $ of
the equation (\ref{le2}) with%
\[
\left\Vert y-x\right\Vert _{\omega}<c\varepsilon.
\]

\end{definition}

Let%
\[
g\left(  t\right)  :=\ ^{RL}D_{-h^{+}}^{\alpha}x\left(  t\right)
-\lambda\ ^{RL}D_{-h^{+}}^{\beta}x\left(  t\right)  -\mu y\left(  t-h\right)
-f\left(  t,x\left(  t\right)  \right)  .
\]
The solution of
\[
^{RL}D_{-h^{+}}^{\alpha}x\left(  t\right)  -\lambda\ ^{RL}D_{-h^{+}}^{\beta
}x\left(  t\right)  -\mu y\left(  t-h\right)  =f\left(  t,x\left(  t\right)
\right)  +g\left(  t\right)  ,
\]
can be represented by%
\[
x\left(  t\right)  =\left(  \mathfrak{F}x\right)  \left(  t\right)  +\int
_{0}^{t}\mathfrak{E}_{\alpha-\beta,\alpha}^{h,\alpha}\left(  \lambda
,\mu;t-s\right)  g\left(  s\right)  ds,\ \ \ \ t\in\left[  0,T\right]  .
\]
Then we have the following estimation.%
\begin{equation}
\left\vert x\left(  t\right)  -\left(  \mathfrak{F}x\right)  \left(  t\right)
\right\vert \leq\int_{0}^{t}\left\vert \mathfrak{E}_{\alpha-\beta,\alpha
}^{h,\alpha}\left(  \lambda,\mu;t-s\right)  \right\vert \left\vert g\left(
s\right)  \right\vert ds\leq\varepsilon T^{\alpha-1}\exp\left(  \left\vert
\lambda\right\vert T^{\alpha-\beta}+\left\vert \mu\right\vert T^{\alpha
}\right)  ,\ \ t\in\left[  0,T\right]  . \label{ul2}%
\end{equation}
Now we are ready to state our Ulam--Hyers stability result.

\begin{theorem}
Assume that (A$_{1}$) is satisfied. Then the equation (\ref{le2}) is
Ulam--Hyers stable on $\left[  0,T\right]  $.
\end{theorem}

\begin{proof}
Let $x\in C\left(  \left[  0,T\right]  ,\mathbb{R}\right)  $ be a solution of
the inequality (\ref{ul1}) and let $y$ be a unique solution of the Cauchy
problem (\ref{le2}), that is,%
\[
y\left(  t\right)  =\left(  \mathfrak{F}y\right)  \left(  t\right)
,\ \ t\in\left[  0,T\right]  .
\]
By using estimation (\ref{qw2}) and inequality (\ref{ul2}), we have%
\begin{align*}
&  \frac{1}{E_{\alpha}\left(  \omega;t\right)  }\left\vert y\left(  t\right)
-x\left(  t\right)  \right\vert \leq\frac{1}{E_{\alpha}\left(  \omega
;t\right)  }\left\vert \left(  \mathfrak{F}y\right)  \left(  t\right)
-\left(  \mathfrak{F}x\right)  \left(  t\right)  -\int_{0}^{t}\mathfrak{E}%
_{\alpha-\beta,\alpha}^{h,\alpha}\left(  \lambda,\mu;t-s\right)  g\left(
s\right)  ds\right\vert \\
&  \leq\frac{\Gamma\left(  \alpha\right)  }{\omega}L_{f}\exp\left(  \left\vert
\lambda\right\vert T^{\alpha-\beta}+\left\vert \mu\right\vert T^{\alpha
}\right)  \left\Vert y-z\right\Vert _{\omega}+\varepsilon T^{\alpha-1}%
\exp\left(  \left\vert \lambda\right\vert T^{\alpha-\beta}+\left\vert
\mu\right\vert T^{\alpha}\right)  .
\end{align*}
By choosing $\omega>L_{f}\Gamma\left(  \alpha\right)  \exp\left(  \left\vert
\lambda\right\vert T^{\alpha-\beta}+\left\vert \mu\right\vert T^{\alpha
}\right)  $ which yields that%
\[
\left\Vert y-z\right\Vert _{\omega}\leq\varepsilon\frac{T^{\alpha-1}%
\exp\left(  \left\vert \lambda\right\vert T^{\alpha-\beta}+\left\vert
\mu\right\vert T^{\alpha}\right)  }{1-\frac{\Gamma\left(  \alpha\right)
}{\omega}L_{f}\exp\left(  \left\vert \lambda\right\vert T^{\alpha-\beta
}+\left\vert \mu\right\vert T^{\alpha}\right)  }.
\]
The proof is completed.
\end{proof}

\end{document}